\documentclass[11pt]{article}
 \usepackage{amsmath,amsthm,amsfonts,amssymb}
\usepackage{verbatim}
\usepackage{latexsym}
\newtheorem{theorem}{Theorem}[section]
\newtheorem{lemma}[theorem]{Lemma}
\newtheorem{proposition}[theorem]{Proposition}
\newtheorem{definition}[theorem]{Definition}
\newtheorem{corollary}[theorem]{Corollary}

\newtheorem{remark}[theorem]{Remark}
\newenvironment{prf} {{\bf Proof.}}{\hfill $\Box$}

\begin{document}
\title{Vanishing moment conditions for wavelet atoms in higher dimensions}
\author{ Hartmut F\"uhr\\
\footnotesize\texttt{{fuehr@matha.rwth-aachen.de}} } \maketitle
\begin{abstract}
We provide explicit criteria for wavelets to give rise to frames and atomic decompositions in ${\rm L}^2(\mathbb{R}^d)$, but also in more general Banach function spaces. We consider wavelet systems that arise by translating and dilating the mother wavelet, with the dilations taken from a suitable subgroup of ${\rm GL}(\mathbb{R}^d)$, the so-called {\em dilation group}.The paper provides a unified approach that is applicable to a wide range of dilation groups, thus giving rise to new atomic decompositions for homogeneous Besov spaces in arbitrary dimensions, but also for other function spaces such as shearlet coorbit spaces. 

The atomic decomposition results are obtained by applying the coorbit theory developed by Feichtinger and Gr\"ochenig, and they can be informally described as follows: Given a function $\psi \in {\rm L}^2(\mathbb{R}^d)$ satisfying fairly mild decay, smoothness and vanishing moment conditions, {\em any} sufficiently fine sampling of the translations and dilations will give rise to a wavelet frame. Furthermore, the containment of the analyzed signal in certain smoothness spaces (generalizing the homogeneous Besov spaces) can be decided by looking at the frame coefficients, and convergence of the frame expansion holds in the norms of these spaces.  We motivate these results by discussing nonlinear approximation. 
\end{abstract}

\noindent {\small {\bf Keywords:} square-integrable group representation; continuous wavelet transform; coorbit spaces; Banach frames; irregular wavelet frames; vanishing moments; nonlinear approximation; shearlets; anisotropic wavelet systems}

\noindent{\small {\bf AMS Subject Classification:} 42C15; 42C40; 46E35}

\section{Introduction}\label{introduction}

%
%

The great success of wavelet theory in applications largely depends on two features: Approximation-theoretic properties of wavelet orthonormal bases, and the availability of fast algorithms in the discrete-domain setting. There are many facets to the approximation-theoretic properties; in this exposition I will focus mainly on a particularly appealing set of results, namely the wavelet characterization of homogeneous Besov spaces. Given a wavelet orthonormal basis $(\psi_{j,k})_{j,k \in \mathbb{Z}} \subset {\rm L}^2(\mathbb{R})$, every function $f \in {\rm L}^2(\mathbb{R})$ has the expansion 
\begin{equation}
\label{eqn:wavelet_onb_exp}
f = \sum_{j,k} \langle f, \psi_{j,k} \rangle \psi_{j,k}~,
\end{equation} with a square-summable coefficient family $(\langle f, \psi_{j,k} \rangle)_{j,k \in \mathbb{Z}}$. More importantly however, (\ref{eqn:wavelet_onb_exp}) converges in smoothness spaces such as the homogeneous Besov spaces $\dot{B}_{p,q}^s(\mathbb{R})$, as soon as $f$ belongs to that space. In fact, the wavelet system is an unconditional basis of that space, and membership of $f \in \dot{B}_{p,q}^s(\mathbb{R})$ is equivalent to a suitable weighted $\ell^{p,q}$-summability condition on the coefficients (thus can be decided just by looking at the coefficients). Thus the wavelet decomposition is valid {\em simultaneously} in a whole range of smoothness spaces, and this observation provides a solid theoretic foundation for the development and study of algorithms for applications like denoising or compression. (See subsection \ref{subsect:nonlin} for a discussion of nonlinear approximation.) 

In this paper, we wish to extend these results to arbitrary dimensions, replacing dyadic dilations in dimension one by certain rather general groups $H$ of (typically nonscalar) dilations; $H$ is called the {\em dilation group}. The rationale for choosing $H$ is explained in more detail in the next section. This is a rather general setting with a lot of different groups to choose from, including the similitude group in arbitrary dimensions, diagonal groups, but also the shearlet group that has received considerable attention in the past few years; see \cite{shearlet_book} for an introduction.  

For arbitrary dilation groups, the existence of associated orthonormal wavelet bases is not established, and we will thus be concerned with frame rather than ONB expansions. The frame expansions will be obtained as discretization of the {\em continuous wavelet transform} associated to the matrix group $H$, by applying the coorbit theory established by Feichtinger and Gr\"ochenig, see \cite{FeiGr0,FeiGr1,FeiGr2,Gr}. The previous paper \cite{Fu_coorbit} established that coorbit theory applies to a large class of dilation groups and their associated continuous wavelet transforms, and thus provides the existence of a consistently defined scale of Besov-type function spaces, atomic decompositions in terms of bandlimited Schwartz functions, etc. Furthermore, \cite{Fu_coorbit} developed vanishing moment criteria for analyzing windows, which mimic the simple criteria for wavelet ONB's mentioned above. 

In principle, coorbit theory provides a description of ``nice'' wavelets that can be used to obtain simultaneous atomic decompositions for a whole range of Besov-type spaces; this was used in \cite{Fu_coorbit} to show the existence of band-limited atoms for this setting. However, both for practical and theoretical considerations, one would also like to replace bandlimited atoms by, say, compactly supported ones. This raises the challenge of finding explicit and easily fulfilled criteria for ``nice'' wavelets with compact support (or, more generally, with certain decay and/or integrability conditions), and it is the chief purpose of this paper to provide such criteria for general dilation groups. 

As a result, we will obtain a very general approach to the construction of wavelet frames in higher dimensions, with very mild conditions on the wavelets to be chosen, and a large class of dilation groups to choose from. The latter aspect is of particular relevance for the ongoing search for anisotropic wavelet systems designed to resolve singularities in dimensions two and higher, such as the shearlet systems. 

While the construction of wavelet frames and bases is by now very well understood (at least for some groups), constructing such systems, and guaranteeing their properties, is still fairly cumbersome. ONB's are typically constructed from a multiresolution analysis, and their existence has been established (to my present knowledge) only for very few classes of dilations, such as dyadic dilations. In the construction of frames, the sampling set used to discretize shifts and dilations is typically fixed beforehand, and then the frame (or ONB) generators are constructed for this particular choice of sampling set, see \cite{BeTa} for the prototype. Furthermore, the constructions for concrete groups, as in the shearlet case, are typically taylored to the specific structure of the dilation group. By contrast, the discretization methods developed in \cite{FeiGr0,FeiGr1,FeiGr2} start from a given ``nice'' wavelet and yield that the action of {\em any} sufficiently dense uniformly discrete sampling set contained in the 
underlying affine group will give rise to a frame. The price one pays for this generality lies in the absence of explicit sampling densities and frame bounds. It is the chief contribution of this paper to provide explicit and easily verifiable criteria that allow to spot a nice wavelet. 

\subsection{Nonlinear approximation using wavelet frames}
\label{subsect:nonlin}

To illustrate the results in our paper, let us take a closer look at the problem of nonlinear approximation. The following discussion loosely follows \cite[Section 4.4]{DaKuStTe}. Let 
$(\psi_\lambda)_{\lambda \in \Lambda} \subset \mathcal{H}$ denote any system of vectors in a Hilbert space $\mathcal{H}$. We define the associated nonlinear approximation error by 
\begin{equation} \label{eqn:approx_error}
 E_n(f; (\psi_\lambda)_{\lambda}) = \inf \left\{ \left\| f - \sum_{\lambda \in \Lambda'} c_\lambda \psi_\lambda \right\|_{\mathcal{H}} ~:~ c_\lambda \in \mathbb{C}~,~|\Lambda'| \le n \right\} ~.\end{equation}

A famous result in wavelet approximation theory states that the nonlinear approximation error of $f \in {\rm L}^2(\mathbb{R})$ in a wavelet orthonormal basis $(\psi_{j,k})_{j,k \in \mathbb{Z}} \subset {\rm L}^2(\mathbb{R})$ fulfills
\[
\sum_{n=1}^\infty n^{-p/2} E_n(f; (\psi_{j,k}))^p < \infty
\] for some $0 \le p< 2$ iff $f \in \dot{B}_{p,p}^{1/p-1/2}(\mathbb{R})$; see e.g. \cite{DVJaPo} for a much more complete discussion. Note that finiteness of the sum can be understood as a sort of decay condition on the approximation error, which becomes sharper as $p$ decreases. 

There are however some conditions that a wavelet has to meet: The precise range of homogeneous Besov spaces for which the above equivalence is valid depends on properties of the wavelet, typically formulated in terms of  decay, smoothness and vanishing moments. However, it should be stressed that these conditions are fairly easy to verify, and the existence of wavelet ONB's fulfilling them to any prescribed (finite) order has been established early on. 

If one replaces bases by frames, the above sharp characterization of nonlinear approximation rate and $p$-summability of the coefficients no longer holds; however, at least one direction remains intact. The proof of the following proposition follows by the exact same reasoning as in \cite[Theorem 4.9]{DaKuStTe}. 
\begin{proposition} \label{lem:nl_approx_frame}
 Let $(\psi_\lambda)_{\lambda \in \Lambda}$ denote a frame in the Hilbert space $\mathcal{H}$. Given $f \in \mathcal{H}$, let 
 \[
 f  =  \sum_{\lambda \in \Lambda} c_\lambda \psi_\lambda
 \] with suitable coefficients $(c_\lambda)_{\lambda \in \Lambda}$. Let $1  \le p < 2$. Then there exists a constant $C$ depending on $p$ and the frame constants such that if $(c_\lambda)_{\lambda \in \Lambda} \in \ell^p(\Lambda)$, one has
 \[
 \left( \sum_{n=1}^\infty n^{-p/2} E_n(f; (\psi_\lambda))^p \right)^{1/p} \le C \| (c_\lambda)_\lambda \|_p~.  
 \]
\end{proposition}

%

Using this observation, we can formulate an application of the results in this paper to homogeneous Besov spaces. The chief difference to the existing results (as far as I know them) is the great freedom in choosing the analysing function: {\em Any} reasonably regular function fulfilling explicit decay and vanishing moment conditions will give rise to a wavelet frame with properties analogous to wavelet ONB's, as long as the sampling in time, scale and rotation is sufficiently fine. For the proof, we refer to Remark \ref{rem:proof_ex_1_2} below. Note that the theorem employs the usual notations $x^j = x_1^{j_1}\ldots x_d^{j_d}$ for vectors $x \in \mathbb{R}^d$ and multiindices $j \in \mathbb{N}_0^d$, as well as $|j| = \sum_{i=1}^d j_i$. 
\begin{theorem} \label{thm:ex_Besov}
Let $k \in \mathbb{N}$. Assume that $\psi \in {\rm L}^2(\mathbb{R}^d) \cap C^{2k}(\mathbb{R}^d)$ fulfills
 \[
\forall 0 \le |j| <k ~:~  \int_{\mathbb{R}^d} x^j \psi(x) dx = 0~,~ 
 \]
 with absolute convergence. Assume further that all partial derivatives of $x^\beta \psi$ of order up to $k$ are integrable, for all multiindices $\beta$ of length $\le k$. 
 
 If $k>\frac{11}{2}d+3$, there exist $\epsilon>0$ and a neighborhood $U \subset SO(d)$ of the identity matrix, both depending only on $\psi$, such that for all $\delta_1,\delta_2< \epsilon$ and all finite sets $\{ h_1,\ldots, h_r \} \subset SO(d)$ satisfying 
 \[
  SO(d) = \bigcup_{\ell=1}^r h_i U~,
 \]
 the wavelet system
  \[
   (\psi_{j,k,\ell})_{j,k \in \mathbb{Z}^d,\ell=1,\ldots, d} ~\mbox{with}~\psi_{j,k,\ell}(t) = (1+\delta_1)^{j/2} \psi((1+\delta_1)^j h_\ell^{-1} t- \delta_2 k)
  \] is a frame for ${\rm L}^2(\mathbb{R}^d)$.
  Furthermore, we find for any $f \in {\rm L}^2(\mathbb{R}^d)$ that $f \in \dot{B}_{p,p}^{d/2-d/p}(\mathbb{R}^d)$ iff 
  the coefficient family $(\langle f, \psi_{j,k,\ell} \rangle)_{j,k \in \mathbb{Z},\ell=1,\ldots,r}$ is $p$-summable. In this case, 
  the approximation error $E_n(f;(\psi_{j,k,\ell})_{j,k,\ell})$ fulfills
 \[  \left( \sum_{n=1}^\infty n^{-p/2} E_n(f;(\psi_{j,k,\ell})_{j,k,\ell})^p \right)^{1/p}\le  C \| f \|_{\dot{B}_{p,p}^{d/2-d/p}}~.\]
\end{theorem}

Observe that is easy to construct functions $\psi$ as in the theorem: Simply pick a function $\rho$ with suitable decay in all derivatives of order up to $2k$, and differentiate $k$ times. If the function $\psi$ is isotropic, the wavelet transform is constant on $SO(d)$-cosets, and the rotations $h_1,\ldots,h_r$ can be omitted from the theorem. Isotropic wavelets are constructed by picking a suitable isotropic function and applying the Laplacian sufficiently often.  

The (relevant) case $0 < p < 1$ is excluded here, chiefly because the results established in this paper and the precursor \cite{Fu_coorbit} are confined to coorbit spaces associated to Banach (rather than quasi-Banach) spaces. The extension to quasi-Banach spaces is the subject of ongoing research. 

 As far as I am aware, the theorem is new, even for dimension one. The sampling set in the theorem was chosen as regular grid mostly for the sake of notational convenience; the underlying sampling theorems due to Feichtinger and Gr\"ochenig theory allow much more general sampling sets, see \ref{thm:at_dec} below. For these irregular sampling sets, the Fourier techniques typically used to derive frame characterizations of homogeneous Besov spaces such as the $\varphi$-transform \cite{FJW} fail, whereas an analog of Theorem \ref{thm:ex_Besov} is still available. 
 
 The literature on irregular wavelet frames contains certain analogs to \ref{thm:ex_Besov}, usually restricted to the ${\rm L}^2$-case. To my knowledge, the paper \cite{FeSu} is closest to the aims and scope of the present article, but it treats only isotropic dilations. Another paper worthwhile mentioning in this context is \cite{SuZh}. By comparison to the results in those papers, the condition $k > \frac{11}{2} d+ 3$ from the theorem seems quite restrictive. Note however that our condition is also sufficient  for anisotropic wavelets, and the theorem points out that the wavelet system generated by isotropically dilating and translating a suitable finite number of rotations of such a wavelet will again yield a Banach frame. 
 
 To some degree, additional restrictions on the wavelets are to be expected from an approach that aims at treating large classes of dilation groups and function spaces other than ${\rm L}^2$ in a unified perspective. As a rule, the explicit conditions on wavelets that are determined in this paper are derived as a proof of principle, and do not have any claim to optimality.  We refer to Remark \ref{rem:suboptimal} for more detailed comments on this issue. 
 
 There is another direction in which the coorbit view provides a considerable extension of the known results: Note that the above statement arises from a particular choice of dilation group, namely the so-called {\em similitude group} $H = \mathbb{R}^+ \cdot SO(d)$. As will be demonstrated in Section \ref{sect:examples}, the results in this paper apply to a much larger variety of groups, including the shearlet dilation group studied, e.g., in \cite{DaKuStTe,DaStTe10,DaStTe11, shearlet_book}. The first compactly supported shearlet frames for coorbit spaces other than ${\rm L}^2(\mathbb{R}^2)$ were introduced fairly recently in \cite{DaStTe11}, and their construction is considered an important step both for applied and theoretical purposes (e.g., for the derivation of trace theorems); see also \cite{KiKuLi} for the related case of cone-adapted shearlet frames. The methods developed in this paper provide an analog of Theorem \ref{thm:ex_Besov} also for the shearlet setting: Here, the pertinent notion of 
vanishing moments is
\begin{equation} \label{eqn:van_mom_shearlet}
\forall  0 \le |j| < k~,~\forall \xi_2 \in \mathbb{R}~:~ \int_{\mathbb{R}^2} x^j  \psi(x) e^{-2 \pi i \langle \xi_2, x_2\rangle} dx =  0~,
\end{equation}
with absolute convergence of the integrals. Any function that possesses sufficiently many vanishing moments and integrable partial derivatives, under the action of a sufficiently dense sampling set, will then give rise to a frame; and it is very easy to produce compactly supported functions fulfilling these conditions. For comparison, \cite[Corollary 3.3]{DaStTe11} imposes Fourier-side decay conditions that in fact follow from regularity of $\psi$ and the vanishing moment conditions (\ref{eqn:van_mom_shearlet}). Thus there is an obvious similarity between the cited result and the criteria derived in this paper; however, \cite[Corollary 3.3]{DaStTe11} in addition requires compact support. 

Again, a priori estimates of the decay of the nonlinear approximation error are available, where the role of the Besov spaces is taken over by the shearlet coorbit spaces $Co ({\rm L^p}(G))$; this fact has already been pointed out in \cite{DaKuStTe}. It turns out that imposing vanishing moments of order $k \ge 127$ (as defined by (\ref{eqn:van_mom_shearlet})), and partial derivatives of to order up to $k$ with suitable decay will allow to formulate a precise shearlet analog of Theorem \ref{thm:ex_Besov}; see Remark \ref{rem:van_mom_shearlet} below. 
(When pondering the fairly astronomical number of 127 vanishing moments, please recall the above disclaimer concerning optimality of the constants derived in this paper.) 

\subsection{A short overview of the paper}

The present paper is a continuation of \cite{Fu_coorbit}. The chief purpose of both papers is to make certain useful but abstract notions from coorbit theory explicit for the concrete case of wavelet systems arising from the action of an affine group generated by the translations and a suitable closed group $H< {\rm GL}(\mathbb{R}^d)$ of dilations. The key object that coorbit theory provides for the construction of atomic decompositions is the space $\mathcal{B}_{v_0} \subset {\rm L}^2(\mathbb{R}^d)$; essentially, this is the set of ``nice'' wavelet for which analogs of Theorem \ref{thm:ex_Besov} can be formulated. The formal definition of this space is fairly technical, see Section \ref{sect:coorbit} below, and it is the chief contribution of this paper to provide concrete and easily verified sufficient criteria for elements of $\mathcal{B}_{v_0}$. 

The paper is structured as follows: Section \ref{sect:coorbit} contains the necessary notions and results concerning continuous wavelet transforms in higher dimensions. Coorbit theory is based on the theory of square-integrable irreducible representations, and for the setup of an affine group acting on ${\rm L}^2(\mathbb{R}^d)$ in the natural way, it is very well understood, which dilation groups provide such representations. To any such group, there exists an associated {\em open dual orbit} $\mathcal{O}$, which will play a crucial role in the paper. Section \ref{sect:coorbit} also contains the necessary ingredients of coorbit theory required, in particular the definition and basic properties of the spaces $Co(Y)$ and the spaces $\mathcal{A}_{v_0}$ and $\mathcal{B}_{v_0}$ of analyzing vectors and ``nice'' wavelets, respectively. 

Section \ref{sect:main} contains the central result of this paper. Theorem \ref{thm:main} contains a sufficient criterion for nice wavelets in terms of vanishing moments. Here the proper notion of vanishing moments is crucial: A function has vanishing moments iff its Fourier transform vanishes to a certain order on the complement $\mathcal{O}^c$ of the open dual orbit. The latter subset can best be understood as the ``blind spot'' of the wavelet transform, containing those frequencies which the wavelet transform cannot resolve well. (This set is of measure zero, hence the existence of the blind spot is no contradiction to the wavelet inversion formula.) 

Theorem \ref{thm:main} does not come entirely for free: There is still one obstacle to its applicability, encapsulated in the notion of {\em strongly temperately embedded dual orbit}, see Definition \ref{defn:str_temperately_embed}. I therefore investigate, for various classes of groups, whether this condition is fulfilled. For all groups that were considered, including diagonal and similitude groups in arbitrary dimensions, as well as the shearlet group and a family of close relatives, the answer is yes. As a result one obtains concrete criteria which considerably generalize the known results. 

\section{Coorbit spaces over general dilation groups} \label{sect:coorbit}

First some notation: 
 Given $f \in {\rm L}^1(\mathbb{R}^d)$, its Fourier transform is defined as
\[
 \mathcal{F}(f)(\xi) := \widehat{f}(\xi) := \int_{\mathbb{R}^d} f(x) e^{-2\pi i \langle x,\xi \rangle} dx~,
\] with $\langle \cdot, \cdot \rangle$ denoting the euclidean scalar product on $\mathbb{R}^d$. We will use the same symbol $\mathcal{F}$ for the Fourier transform of tempered distributions.
 For any subspace $X \subset
\mathcal{S}'(\mathbb{R}^d)$, we let $\mathcal{F}^{-1} X$ denote its
inverse image under the Fourier transform.

 In order to avoid cluttered notation, we will occasionally use
the symbol $ X \preceq Y$ between expressions $X,Y$ involving one or
more functions or vectors in $\mathbb{R}^d$,  to indicate the existence of a constant $C>0$,
independent of the functions and vectors occurring in $X$ and $Y$, such that
 $X \le CY$. We let $|\cdot|: \mathbb{R}^d \to
\mathbb{R}$ denote the euclidean norm.  Given a matrix $h \in \mathbb{R}^{d \times d}$, the operator norm
of the induced linear map $(\mathbb{R}^d, |\cdot|) \to  (\mathbb{R}^d, |\cdot|)$ is denoted by
 $\| h \|_\infty$. By a slight abuse of notation we will also use $|\alpha| = \sum_{i=1}^d \alpha_i$ for multiindices $\alpha \in \mathbb{N}_0^d$.  

For $r,m>0$, we let
\[
| f |_{r,m} = \sup_{x \in \mathbb{R}^d, |\alpha| \le r} (1+|x|)^{m}
|\partial^\alpha f (x)|~.
\] denote the associated Schwartz norm of a function $f: \mathbb{R}^d \to \mathbb{C}$ with suitably many partial derivatives.

Let us now describe the necessary notions connected to continuous wavelet transforms. 
We fix a closed matrix group $H < {\rm
GL}(d,\mathbb{R})$, the so-called {\bf dilation group}, and let $G =
\mathbb{R}^d \rtimes H$. This is the group of affine
mappings generated by $H$ and all translations. Elements of $G$ are
denoted by pairs $(x,h) \in \mathbb{R}^d \times H$, and the product
of two group elements is given by $(x,h)(y,g) = (x+hy,hg)$. 
The left Haar measure of $G$ is given by $d(x,h) = |\det(h)|^{-1} dx dh$, and the modular
function of $G$ is given by $\Delta_G(x,h) = \Delta_H(h)
|\det(h)|^{-1}$.

$G$ acts unitarily on ${\rm L}^2(\mathbb{R}^d)$ by the {\bf quasi-regular
representation} defined by
\begin{equation} \label{eqn:def_quasireg}
[\pi(x,h) f](y) = |{\rm det}(h)|^{-1/2} f\left(h^{-1}(y-x)\right)~.
\end{equation}
We assume that $H$ is chosen such that $\pi$ is an {\bf (irreducible) square-integrable
representation}. Square-integrability of the representation means
that there exists at least one nonzero {\bf admissible vector} $\psi
\in {\rm L}^2(\mathbb{R}^d)$ such that the matrix coefficient
\[
(x,h) \mapsto \langle \psi, \pi(x,h) \psi \rangle
\] is in ${\rm L}^2(G)$, which is the ${\rm L}^2$-space associated
to a left Haar measure on $G$. In this case the
associated wavelet transform
\[
\mathcal{W}_\psi : {\rm L}^2(\mathbb{R}^d) \ni f \mapsto \left(
(x,h) \mapsto \langle f, \pi(x,h) \psi \rangle \right)
\] is a scalar multiple of an isometry, which gives rise
to the {\bf wavelet inversion formula}
\begin{equation} \label{eqn:wvlt_inv}
f = \frac{1}{c_\psi} \int_G \mathcal{W}_\psi f(x,h) \pi(x,h) \psi ~
d\mu_G(x,h)~.
\end{equation}

A thorough understanding of the properties of the wavelet transform hinges on the {\em dual action}, i.e., 
 the (right) linear action $\mathbb{R}^d \times H \ni(\xi,h) \mapsto
 h^T \xi$: By the results of \cite{Fu96,Fu10}, $H$ is admissible iff the dual action has a single open orbit
 $\mathcal{O} = \{ h^T \xi_0 : h \in H \} \subset \mathbb{R}^d$ of full measure (for some $\xi_0 \in \mathcal{O}$), such
 that in addition the stabilizer group $H_{\xi_0} = \{ h \in H : h^T \xi_0 = \xi_0 \}$ is
 compact. (This condition does of course not depend on $\xi_0 \in
 \mathcal{O}$.) The dual orbit will be of central importance to this paper.
 
 Let us next describe the pertinent notions from coorbit theory. 
 A weight on a locally compact group $K$ is a continuous function $w: K \to \mathbb{R}^+$ satisfying $w(xy) \le w(x) w(y)$, for all $x,y \in K$. 
The Besov-type coorbit spaces that we focus on in this paper are obtained by fixing a weight $v$ of the type 
\begin{equation} \label{eqn:defn_v} v(x,h) =  (1+|x| + \|h\|_\infty)^s w(h) \end{equation} on $G$, where
$|\cdot|: \mathbb{R}^d \to \mathbb{R}$ is an arbitrary fixed norm, and $w$ is some weight on $H$. (Note that this indeed defines a weight $v$.) We then define, for $1 \le p,q < \infty$,
\[
L^{p,q}_v (G) = \left\{ F: G \to \mathbb{C}~:~\int_H \left(
\int_{\mathbb{R}^d} |F(x,h)|^p v(x,h)^p dx \right)^{q/p} \frac{dh}{|{\rm det}(h)|} < \infty
\right\}~,
\]
with the obvious norm, and the usual conventions regarding
identification of a.e. equal functions. 

We write ${\rm L}^p_v(G) = {\rm L}^{p,p}_v(G)$. The corresponding spaces for $p=\infty$ and/or
$q = \infty$ are defined by replacing integrals with essential
suprema. We will also use
\[
{\rm L}^p_s(\mathbb{R}^d) = \left\{ f \mbox{
Borel-measurable}~:~\int_{\mathbb{R}^d} |f(x)|^p (1+|x|)^{sp} dx <
\infty \right\}~.
\]
 
 We next recall the necessary ingredients of coorbit theory. Our main sources for the following are the papers \cite{FeiGr0,FeiGr1,FeiGr2,Gr}. We assume that $Y$
is a Banach space of functions on $G$ that fulfills the conditions of \cite[2.2]{Gr}, i.e. it is continuously embedded in ${\rm L}^1_{loc}(G)$, and
fulfills certain compatibility conditions with convolution. Examples of such spaces are the 
 ${\rm L}^{p,q}_v(G)$ defined above. The following definition will be important:
\begin{definition}
A weight $v_0$ is called {\bf control weight} for $Y$ if it satisfies 
\[
 v_0(x,h) = \Delta_G(x,h)^{-1} v_0((x,h)^{-1})~, 
\]
as well as 
 \[
\max \left( \|L_{(x,h)^{\pm 1}} \|_{Y \to Y},\| R_{(x,h)} \|_{Y \to Y},\|
R_{(x,h)^{-1}} \|_{Y \to Y} \Delta_G(x,h)^{-1} \right) \le v_0(x,h)
\] where $L_{(x,h)},R_{(x,h)}$ are left and right translation by $(x,h) \in G$.
\end{definition}

Using a control weight $v_0$ for $Y$ we define the set
\[ \mathcal{A}_{v_0} = \{ \psi \in {\rm L}^2(\mathbb{R}^d)~:~\mathcal{W}_{\psi}\psi \in {\rm L}^1_{v_0}(G) \} \] of analyzing vectors. It turns out that $\mathcal{A}_{v_0}$ is a vector space, in fact a Banach space, and invariant under $\pi$. We denote its conjugate dual as $\mathcal{A}_{v_0}^\sim$. 
The sesquilinear map ${\rm L}^2(\mathbb{R}^d) \times \mathcal{A}_{v_0}
\ni (f,g) \mapsto \langle f,g \rangle$ can be uniquely extended to
$\mathcal{A}_{v_0}^{\sim} \times \mathcal{A}_{v_0}$. Hence, if we
fix $\psi \in \mathcal{A}_{v_0}$, the definition of 
the continuous wavelet transform of $f \in \mathcal{A}_{v_0}^{\sim}$
via
\[
 \mathcal{W}_\psi f(x,h) = \langle f, \pi(x,h) \psi \rangle
\] again makes sense. 

Now the coorbit space associated to
$Y$ is defined by fixing a nonzero  $\psi \in \mathcal{A}_{v_0}$ and letting 
\[
{\rm Co}(Y) = \{ f \in \mathcal{A}_{v_0}^{\sim} : \mathcal{W}_\psi f
\in Y \}
\] with the norm $\| f \|_{{\rm Co}(Y)} = \| \mathcal{W}_\psi f
\|_Y$. It now follows by \cite[Theorem 5.2]{FeiGr1} that the space $Co Y$ is a Banach space, and independent of the choice of the analyzing vector, as well as of the precise choice of control weight.  

Clearly, the whole construction hinges on the actual existence of a nonzero analyzing vector. For arbitrary control weights $v_0$, this might be difficult to answer. For the space $Y = {\rm L}^{p,q}_v(G)$, with $v$ as in (\ref{eqn:defn_v}),  we first note that by \cite[Lemma 2.3]{Fu_coorbit}, 
there exists a control weight $v_0$ for $Y = {\rm L}^{p,q}_v(G)$ satisfying the estimate
 \begin{equation} \label{eqn:cont_weight_sep}
  v_0(x,h) \le (1+|x|)^s w_0(h)~,
 \end{equation} with $w_0 : H \to \mathbb{R}^+$ defined by 
 \begin{eqnarray*} \nonumber
  w_0(h) & = &    (w(h)+w(h^{-1})) \max \left(\Delta_G(0,h)^{-1/q}, \Delta_G(0,h)^{1/q-1} \right)  \\  & & \times \left(|{\rm det}(h)|^{1/q-1/p} + |{\rm det}(h)|^{1/p-1/q} \right) (1+\|h\|_\infty+\|h^{-1}\|_\infty)^s~.
 \end{eqnarray*} 
 But then Theorem 2.1 of \cite{Fu_coorbit} implies that $\mathcal{A}_{v_0}$ is nontrivial, and thus $Co Y$ is indeed well-defined.

 We next turn to the study of atomic decompositions. As already explained in the introduction, the sampling theorems derived in \cite{FeiGr0,FeiGr1,FeiGr2,Gr} are rather flexible in terms of the sampling sets, at the price of imposing additional conditions on the analyzing vector. These are formulated in the next definition. 
 
 \begin{definition}
  Let $Y$ denote any solid Banach function space on the locally compact group $G$, $U \subset G$ a compact neighborhood of the identity, and $F: G \to \mathbb{C}$. We let
  \[
   \left(\mathcal{M}_U^R F \right) (x) = \sup_{y \in U}|f(xy)|
  \] denote the local maximum function of $F$ with respect to $U$. Given a weight $v_0$ on $G$, we denote the associated Wiener amalgam space by
  \[
   W^R(C^0,{\rm L}^1_{v_0}) = \{ F : G \to \mathbb{C}~:~ F \mbox{ continuous }, \mathcal{M}_U^R F \in {\rm L}^1_{v_0}(G) \}~,
\]
with norm $\| F \|_{ W^R(C^0,{\rm L}^1_{v_0})} = \|  \mathcal{M}_U^R F \|_{{\rm L}^1_{v_0}}$.

 We let
\[
 \mathcal{B}_{v_0} = \{ \psi \in {\rm L}^2(\mathbb{R}^d)~:~ \mathcal{W}_\psi \psi \in W^R(C^0,{\rm L}^1_{v_0})  \}~. 
\]
 \end{definition}
  
Since ${\rm L}^1_{v_0}(G) \supset  W^R(C^0,{\rm L}^1_{v_0}) $, every nonzero element in $\mathcal{B}_{v_0}$ can be used to characterize elements of $Co Y$, whenever $v_0$ is a control weight for $Y$. The additional condition will allow to derive the desired discrete characterizations as well. In other words, the elements of $\mathcal{B}_{v_0}$, for a control weight $v_0$ associated to a Banach function space $Y$, are precisely the ``nice'' wavelets associated to the coorbit space $Co(Y)$ that were mentioned in the introduction. 

\begin{definition}
 Let $U \subset G$ denote a neighborhood of the identity, and $Z=(z_i)_{i \in I} \subset G$.
\begin{enumerate}
 \item[(a)] The family $(z_i)_{i \in I}$ is called {\bf $U$-dense}, if $\bigcup_{i \in I} z_i U = G$.
 \item[(b)] The family $(z_i)_{i \in I}$ is called {\bf $U$-separated}, if $z_i U \cap z_j U = \emptyset$, whenever $i \not=j$. It is called {\bf separated}, if there exists a neighborhood $U$ of unity such that it is $U$-separated. It is called {\bf relatively separated} if it is the finite union of separated families.
\end{enumerate}
\end{definition}

\begin{remark} \label{rem:ex_rel_sep}
Note that $U$-dense, relatively separated families always exist, for every neighborhood $U$ of the identity. 
More precisely, given any separated family $Z_0$, there exists a $U$-dense separated set $Z$ containing $Z_0$.
To see this, pick a symmetric neighborhood $V$ of the identity such that $Z_0$ is $V$-separated, and in addition fulfills $V^2 = \{ vw : v,w \in V \} \subset U$, and apply Zorn's Lemma to find a $V$-separated family $Z=(z_i)_{i \in I}$ containing $Z_0$ that is maximal with respect to inclusion. Then maximality and $V^2 \subset U$ imply that this family is $U$-dense. 
\end{remark}

For the formulation of the atomic decomposition result, we need a norm on the sequences space. We use $\mathbf{1}_W: G \to \mathbb{R}$ to denote the indicator function of a subset $W \subset G$. 
\begin{definition}
 Let $Y$ be a solid Banach function space on $G$, and $Z \subset G$ a relatively separated set. Picking an arbitrary compact neighborhood $W$ of the identity, we define
 \[
  \| (c_z)_{z \in Z} \|_{Y_d} = \left\| \sum_{z \in Z} |c_z| \mathbf{1}_{z W} \right\|_{Y}~,
 \]
and let $Y_d = \{ (c_z)_{z \in Z} \in \mathbb{C}^{Z}: \| (c_z)_{z \in Z} \|_{Y_d}< \infty \}$.
\end{definition}

\begin{remark} \label{rem:coeff_space}
Note that $Y_d$ is a nontrivial Banach space, and the norm of $Y_d$ is (up to equivalence) independent of the choice of $W$ \cite{FeiGr0}.
For the semidirect product group $G = \mathbb{R}^d \rtimes H$, the typical sampling sets are of the type
\[
 Z = \{ (h_jx_k, h_j): j \in J , k\in K  \}
\] where $\{ h_j : j \in J \} \subset H$ and $\{ x_k : k \in K \} \subset \mathbb{R}^d$ are uniformly dense and separated. 
One easily verifies that $Z$ is uniformly dense and separated in $G$. Indeed, if $V \subset \mathbb{R}^d$ and $W \subset H$ separate $(x_k)_{k \in K}$ and $(h_j)_{j \in J}$, respectively, then $U = \{(x,g): x \in V, g \in W \}$ separates $Z$: $(h_j x_k, h_j) U \cap (h_{j'} x_{k'}, h_{j'}) U \not= \emptyset$ entails $h_j W \cap h_{j'} W \not= \emptyset$, and thus by assumption on $W$, $j=j'$. But then a comparison of the translation coordinates yields $h_j x_k + h_j V \cap h_j x_{k'} + h_j V \not= \emptyset$, i.e., $x_k + V \cap x_{k'}+ V \not= \emptyset$, and thus $k=k'$. A similar calculation shows that $Z$ is $U$-dense if $\{ x_k : k \in K \}$ is $V$-dense and $\{ h_j: j \in J \}$ is $W$-dense.   

If $Y = {\rm L}^{p,q}_v(G)$, for some weight $v$, the associated coefficient space norm is equivalent to a discrete weighted $\ell^{p,q}$-norm, i.e.
\begin{equation} \label{eqn:discrete_norm_equiv}
 \left\| (c_{j,k}) \right\|_{Y_d} \asymp \left( \sum_{j \in J} |{\rm det}(h_j)|^{q/p-1}\left( \sum_{k \in K} \left( |c_{j,k}| v(h_j x_k, h_j) |{\rm det}(h_j)|^{1/p-1/q} \right)^p \right)^{q/p} \right)^{1/q}
\end{equation} with the usual modifications for $p= \infty$ and/or $q=\infty$. To see this, note that we can pick a compact neighborhood $W= W_1 \times W_2 \subset \mathbb{R}^d \times H$ of the neutral element such that $(h_j x_k, h_j) W \cap (h_{j'} x_{k'}, x_{k'}) W = \emptyset$.
By submultiplicativity and continuity of the weight $w$, there are constants $c_1,c_2$ such that
\[
\forall j\in J \forall k \in K~:~  c_1 v(h_j x_k, h_j) \le v|_{(h_j x_k, h_j) W} \le c_2 v(h_j x_k, h_j)~.
\] Hence if we employ this particular set $W$, left-invariance of Haar measure on $G$ ensures  (\ref{eqn:discrete_norm_equiv}).  In particular, the finitely supported sequences are dense in $Y_d$. 

As a special case, we obtain that the discrete coefficient space associated to $Y = {\rm L}^p(G)$ is indeed $\ell^p(Z)$. 
\end{remark}

We can now formulate the atomic decomposition result \cite{Gr}, which will be seen to imply Theorem \ref{thm:ex_Besov}. Note in particular that, since ${\rm L}^2(\mathbb{R}^d) = Co {\rm L}^2(G)$  by \cite[Corollary 4.4]{FeiGr1}, this statement will also yield criteria for wavelet frame generators of ${\rm L}^2(\mathbb{R}^d)$.  
 \begin{theorem} \label{thm:at_dec}
 Let $v_0$ denote a control weight for $Y$, and 
let $0 \not= \psi \in \mathcal{B}_{v_0}$. Assume that the finitely supported seuences are dense in $Y_d$. Then
there exists a neighborhood $U \subset G$ of unity such that for all
$U$-dense, relatively separated families $(z_i)_{i \in I} \subset
G$, the following statements are true:
\begin{enumerate}
 \item[(a)] There is a linear bounded map $C:  Co Y \to Y_d(Z)$ with the property that, for all $f \in Co Y$,
 \[
  f = \sum_{i \in I} C(f) (z_i) \pi(z_i) \psi~,
 \] with unconditional convergence in $\| \cdot \|_{Co Y}$.
\item[(b)] Conversely, for every sequence $(c(z_i))_{i \in I} \in Y_d(Z)$, the sum
\[
 g = \sum_{i \in I} c(z_i) \pi(z_i) \psi
\] converges unconditionally in $\| \cdot \|_{Co Y}$, with $\| g \|_{Co Y} \preceq \| (c(z_i))_{i \in I} \|_{Y_d(Z)}$.
 \item[(c)] The norms $\| f \|_{Co Y}$ and $\| (\mathcal{W}_\psi f (z_i))_{i \in I}\|_{Y_d(Z)}$ are equivalent. Moreover, 
$ f \in {Co Y}$ iff $ (\mathcal{W}_\psi f (z_i))_{i \in I} \in Y_d(Z)$.
\end{enumerate}
\end{theorem}

We next exhibit a class of nice wavelets: Any bandlimited Schwartz function with Fourier support contained in the open dual orbit is in $\mathcal{B}_{v_0}$. This was shown in \cite{Fu_coorbit}, using oscillation estimates. The connection between oscillation and Wiener amalgam space is explained in the next remark:
\begin{remark}
 Let $U \subset G$ denote a relatively compact neighborhood of the identity, and $F: G \to \mathbb{C}$ any function. We let
\[
 {\rm osc}_U(F)(x) = \sup \{ |F(x)-F(xy)|: y \in U \}~.
\] It is easy to see that $F \in  W^R(C^0,{\rm L}^1_{v_0})$ holds iff both $F$ and 
$ {\rm osc}_U(F)$ are in ${\rm L}^1_{v_0}(G)$. 
\end{remark}

The following is \cite[Lemma 2.6]{Fu_coorbit}.
\begin{theorem} \label{thm:bl_atoms} For any weight $v_0$ satisfying $v_0(x,h) \le (1+|x|+\| h \|_\infty)^s w_0(h)$, with $w_0 : H \to \mathbb{R}^+$ an arbitrary weight, we have 
 $\mathcal{F}^{-1}(C_c^\infty(\mathcal{O})) \subset \mathcal{B}_{v_0}$. 
\end{theorem}

\section{Formulation and proof of the main result} \label{sect:main}

The results of the previous section have set the stage for the main result. While bandlimited Schwartz functions are fairly convenient to work with for many purposes, the usefulness of compactly supported wavelets has been emphasized repeatedly. The chief aim of this section is to replace the assumption of proper confinement of the supports on the Fourier transform by a more quantitative version in terms of decay properties of the Fourier transform $\widehat{\psi}(\xi)$, as $\xi$ approaches the boundary of the dual orbit. These decay conditions will be formulated by vanishing moment conditions, defined as follows:
\begin{definition} \label{defn:van_mom}
 Let $r \in \mathbb{N}$ be given.
 $f \in {\rm L}^1(\mathbb{R}^d)$ {\bf has vanishing moments in $\mathcal{O}^c$ of order $r$}  if all
 distributional derivatives $\partial^\alpha \widehat{f}$ with $|\alpha|\le r$ are
 continuous functions, and all derivatives of degree $|\alpha|<r$ are identically vanishing on $\mathcal{O}^c$.
\end{definition}
Note that under suitable integrability conditions on $\psi$, the vanishing moment conditions are equivalent to 
\[
 \forall |j| < k,\forall \xi \in \mathcal{O}^c ~:~\int_{\mathbb{R}^d} x^j \psi(x) e^{-2 \pi i \langle \xi, x \rangle} dx = 0~. 
\]

We next define an auxiliary function $A: \mathcal{O} \to \mathbb{R}^+$ as follows:
Given any point $\xi \in \mathcal{O}$, let ${\rm
dist}(\xi,\mathcal{O}^c)$ denote the minimal distance of $\xi$ to
$\mathcal{O}^c$, and define
\[
 A(\xi) = \min \left( \frac{{\rm dist}(\xi,\mathcal{O}^c)}{1+\sqrt{|\xi|^2-{\rm dist}(\xi,\mathcal{O}^c)^2}},
 \frac{1}{1+|\xi|} \right)~. 
\]
By definition, $A$ is a continuous function with $A(\cdot) \le 1$.

Using $A$, we then define $\Phi_\ell : H \to \mathbb{R}^+ \cup \{ \infty \}$, for $\ell \in \mathbb{N}$ as
\begin{equation} \label{eqn:def_Phi_ell}
 \Phi_\ell(h) =  \int_{\mathbb{R}^d} A(\xi)^\ell A(h^T \xi)^\ell d\xi
\end{equation}
%
%

We note a few simple but useful properties of $\Phi_\ell$: 
\begin{lemma}
 \label{lem:Phi_ell_elementary}
 \begin{enumerate}
  \item[(a)] $\ell \le \ell'$ implies $\Phi_\ell(h) \ge \Phi_{\ell'}(h)$, for all $h \in H$.
  \item[(b)] $\Phi_\ell(h) = |{\rm det}(h)|^{-1} \Phi_\ell(h^{-1})$. 
 \end{enumerate}
\end{lemma}
\begin{proof}
 Part (a) follows from $A(\cdot) \le 1$, and part (b) by substitution: Letting $\omega = h^T \xi$, and using the notation $h^{-T} = (h^{-1})^T$, we find
 \begin{eqnarray*}
 \Phi_\ell(h) & = & \int_{\mathbb{R}^d} A(\xi)^\ell A(h^T \xi)^\ell d\xi \\
 & = & |{\rm det}(h)|^{-1} \int_{\mathbb{R}^d} A(h^{-T} \omega)^\ell A(\omega)^\ell d\omega = |{\rm det}(h)|^{-1} \Phi_\ell(h^{-1})~. 
 \end{eqnarray*}
\end{proof}

Now the following definition will allow to formulate sufficient vanishing moment criteria for elements of $\mathcal{B}_{w_0}$.
\begin{definition} \label{defn:str_temperately_embed}
 Let $w_0: H \to \mathbb{R}^+$ denote a weight, $s \ge 0$. We call $\mathcal{O}$ {\bf strongly $(s,w_0)$-temperately embedded (with index $\ell \in \mathbb{N}$)} if $\Phi_\ell \in W(C^0,{\rm L^1}_{m})$, where the weight $m: H \to \mathbb{R}^+$
is defined by
\[
m(h) = w_0(h) |{\rm det}(h)|^{-1/2} (1+\| h \|_\infty)^{2(s+d+1)} ~.
\]
\end{definition}

This definition plays a similar role as the notion of {\em temperately embedded orbits} introduced in \cite{Fu_coorbit}. Essentially, the index $\ell$ will determine the number of vanishing moments needed to ensure that a given function is in $\mathcal{B}_{w_0}$. Note that somewhat contrary to the intuition conveyed by the terminology, it is currently not clear whether strong temperate embeddedness implies temperate embeddedness. 
\begin{theorem} \label{thm:main}
Let the weight $v_0$ on $G$ fulfill the estimate $v_0(x,h) \le (1+|x|)^s w_0(h)$, and 
assume that $\mathcal{O}$ is strongly $(s,w_0)$-embedded with index $\ell$. 
Then any function $\psi \in
{\rm L}^1(\mathbb{R}^d) \cap C^{\ell+d+1}(\mathbb{R}^d)$ with
vanishing moments in $\mathcal{O}^c$ of order $t>\ell+s+d+1$ and $|\widehat{\psi}|_{t,t}<\infty$
is contained in $\mathcal{B}_{w_0}$. 

There exist compactly supported functions $\psi$ satisfying this condition. 
\end{theorem}

Note that the condition $|\widehat{\psi}|_{t,t} < \infty$ is guaranteed by integrability of $x^\beta \partial^\alpha \psi$, for all multiindices $\alpha,\beta$ of length $\le t$. It is thus fulfilled by all compactly supported $\psi$ with continuous partial derivatives up to order $t$. 

Before we prove the theorem, we need various auxiliary results.
 \begin{lemma} \label{lem:max_fun_weight}
 Let $m$ denote an arbitrary weight on a locally compact group $G$, $f$ a continuous function on $G$, and $U$ a compact neighborhood of the identity. Then there are constants $c_1,c_2$, depending only on $U$ and $m$, such that
 \[
  c_1 \mathcal{M}^R_U (m \cdot f) (x) \le m(x) \mathcal{M}^R_U (f)(x) \le c_2 \mathcal{M}^R_U (m \cdot f)(x)~. 
 \]
\end{lemma}
\begin{prf}
 Using submultiplicativity and continuity of $m$, one readily verifies the estimates with 
 \[
  c_1 = \frac{1}{\sup_{y \in U} m(y)}~,~ c_2 = \sup_{y \in U} m(y^{-1})~. 
 \]
\end{prf}

We next cite a useful  result \cite[Lemma 3.6]{Fu_coorbit}, which explains the usefulness of the auxiliary functions: They serve as envelope functions for the Fourier transform $\widehat{\psi},\widehat{f}$ and their derivatives, which can be translated to decay estimates for wavelet coefficients. 
\begin{lemma} \label{lem:decay_est_prod}
 Let $\alpha \in \mathbb{N}_0^d$ be a multiindex with $|\alpha| < r$. 
Assume that $f,\psi \in {\rm L}^1(\mathbb{R}^d)$ have vanishing
moments of order $r$ in $\mathcal{O}^c$, and fulfill
$|\widehat{f}|_{r,r-|\alpha|} < \infty, |\widehat{\psi}|_{r,r-|\alpha|} < \infty$.Then there exists a
constant $C>0$, independent of $f$ and $\psi$, such that
\[
 | \partial^\alpha (\widehat{f} \cdot D_h \widehat{\psi})(\xi)| \le
 C |\widehat{f}|_{r,r-|\alpha|} |\widehat{\psi}|_{r,r-|\alpha|} 
(1+\| h \|_\infty)^{|\alpha|}
 A(\xi)^{r-|\alpha|} A(h^T \xi)^{r-|\alpha|}~.
\]
Here we used the notation $D_h \widehat{\psi}: \xi \mapsto \widehat{\psi}(h^T \xi)$. 
\end{lemma}

Now the following estimate reveals the usefulness of the auxiliary function $\Phi_\ell$, by 
giving an estimate of the wavelet coefficient $\mathcal{W}_\psi \psi$. The chief advantage of this estimate is that it separates the translation and dilation variables. Note that the same estimate applies also to $\mathcal{W}\psi f$, as long as both $\psi$ and $f$ fulfill the conditions of the Lemma. 
\begin{lemma} \label{lem:decay_est_wc}
 Let $0 < m < r$, and let $\psi \in {\rm L}^1(\mathbb{R}^d)$ denote a function with vanishing moments of order $r$ in $\mathcal{O}^c$ and $|\widehat{\psi}|_{r,r} < \infty$. Then 
 \[
 \left|  \mathcal{W}_\psi \psi (x,h) \right| \preceq |\widehat{\psi}|_{r,r}^2 (1+|x|)^{-m}|\det(h)|^{1/2} (1+\| h \|_\infty)^m \Phi_{r-m}(h)~. 
 \]
\end{lemma}
\begin{prf} Since 
 \[
\mathcal{W}_{\psi} \psi (x,h) = |{\rm det}(h)|^{1/2} \left(  \widehat{\psi}
\cdot (D_h \overline{\widehat{\psi}}) \right)^{\vee}(x)~, 
\] we can use the standard estimate of decay on the space side by ${\rm L}^1$-norms of derivatives on the Fourier transform side to obtain 
\begin{eqnarray*}
 \left| \mathcal{W}_{\psi} \psi (x,h) \right| & \preceq & |{\rm det}(h)|^{1/2} (1+| x|)^{-m} \sum_{|\alpha| \le m} \left\| 
 \partial^\alpha \left(\widehat{\psi} \cdot D_h \overline{\widehat{\psi}} \right) \right\|_1 \\
 & \preceq &  (1+|x|)^{-m} |{\rm det}(h)|^{1/2} \sum_{|\alpha| \le m} |\widehat{\psi}|_{r,r-|\alpha|}^2 (1+\| h \|_{\infty})^{|\alpha|}  \underbrace{\int_{\mathbb{R}^d} A(\xi)^{r-|\alpha|} 
 A(h^T\xi)^{r-|\alpha|} d\xi}_{= \Phi_{r-|\alpha|}(h)} ~, 
\end{eqnarray*}
where we used Lemma \ref{lem:decay_est_prod}. Now the monotonicity properties of the Schwartz norm and the $\Phi_\ell$ (with respect to their subscripts) allow to estimate
\[
  \sum_{|\alpha| \le m} |\widehat{\psi}|_{r,r-|\alpha|}^2 (1+\| h \|_{\infty})^{|\alpha|} \Phi_{r-|\alpha|}(h)  \preceq  |\widehat{\psi}|_{r,r}^2  (1+\| h \|_{\infty})^{m} \Phi_{r-m}(h)~,
\]
which finishes the proof. 
\end{prf}

{\bf Proof of \ref{thm:main}:}
We fix $V = B_1(x)$ and $W = \{ h \in H : \| h-{\rm id} \|_\infty < 1/2\}$. Then $U = V \times W \subset G$ is a neighborhood of the identity in $G$, and we will use this neighborhood to show finiteness of the amalgam norm $\left\| \mathcal{W}_{\psi} \psi \right\|_{W^R(C^0,{\rm L}^1_{v_0})}$. To this end, let $k= s+d+1$, and use Lemma \ref{lem:decay_est_wc} to estimate as follows: 
\begin{eqnarray}
\nonumber \lefteqn{\left\| \mathcal{W}_{\psi} \psi \right\|_{W^R(C^0,{\rm L}^1_{v_0})}} \\ \nonumber & \le & 
 \int_H \int_{\mathbb{R}^d}  \sup_{y \in V, g \in W} \left| \mathcal{W}_\psi \psi (x+hy,hg) \right| (1+|x|)^s dx ~w_0(h) \frac{dh}{|{\rm det}(h)|} \\ \nonumber
 &  \preceq & \left|\widehat{\psi} \right|_{t,t}^2 \int_H \int_{\mathbb{R}^d} \sup_{y \in V, g \in W} \left(  (1+|x+hy|)^{-k} (1+|x|)^s  
 \vphantom{\Phi_{t-k}(hg)} \right.
 \\ \nonumber
 & &  \left. (1+\| h g \|_\infty)^k |{\rm det}(hg)|^{1/2} \int_{\mathbb{R}^d} A(\xi)^{t-k} A((hg)^T \xi)^{t-k} d\xi \right) dx~ w_0(h) \frac{dh}{|{\rm det}(h)|} 
 \\ & \label{eqn:zwischen}  \le & \left|\widehat{\psi} \right|_{t,t}^2 \int_H \int_{\mathbb{R}^d} \left( \sup_{y \in V}  (1+|x+hy|)^{-k}  \right) (1+|x|)^s dx
 \left( \sup_{g \in W} \Psi(hg) \right) w_0(h) \frac{dh}{|{\rm det}(h)|}~,
\end{eqnarray}
where we introduced the auxiliary function
\[
 \Psi(h) = (1+\| h \|_\infty)^k |{\rm det}(h)|^{1/2}\Phi_{t-k}(h) ~. 
\]
Next, using that $V$ is the unit ball, we find $|x+hy| \ge \max(|x| - |hy|,0) \ge \max(|x|-\| h \|_\infty,0)$, and thus  
\[
 \sup_{y \in V} (1+ |x+hy|)^{-k}  \le \left (1+\max(0, |x|-\|h\|_\infty) \right)^{-k}~.
\]
Hence 
\begin{eqnarray*}
\lefteqn{ \int_{\mathbb{R}^d} \left( \sup_{y \in V}  (1+|x+hy|)^{-k} (1+|x|)^s \right) dx} \\
& \le &  \lefteqn{ \int_{\mathbb{R}^d}   \left (1+\max(0, |x|-\|h\|_\infty) \right)^{-k} (1+|x|)^{s} dx} \\ & = &
\int_{|x|\le \| h \|_\infty} (1+|x|)^s dx + \int_{|x|>\|h \|_\infty} (1+|x|-\| h \|_\infty)^{-k} (1+|x|)^s dx \\
& \le & C_1 (1+\| h \|_\infty)^s \| h \|_\infty^d + \int_{\mathbb{R}^d} (1+\left|~ |x|-\|h \|_\infty~ \right| )^{-k}(1+|x|)^s dx~,
\end{eqnarray*} with $C_1$ denoting the volume of the unit ball.
Now the estimate $\left(1+\left| ~|x|-\|h\|_\infty \right| \right) (1+\|h \|_\infty) \ge (1+|x|)$ yields
\[
\int_{\mathbb{R}^d} (1+\left| ~|x|-\|h \|_\infty ~\right|)^{-k}(1+|x|)^s dx  \le (1+\| h \|_\infty)^k \underbrace{\int_{\mathbb{R}^d} (1+|x|)^{s-k} dx}_{=: C_2}
\] with finite constant $C_2$ since $k-s=d+1$, and we finally arrive at 
\begin{equation}
 \label{eqn:inner_estimate}
 \int_{\mathbb{R}^d} \left( \sup_{y \in V}  (1+|x+hy|)^{-k}\right) (1+|x|)^s  dx \le C (1+\|h \|_\infty)^k~. 
\end{equation}
We can now wrap up the proof: Plugging (\ref{eqn:inner_estimate}) into (\ref{eqn:zwischen}), and applying 
Lemma \ref{lem:max_fun_weight} to the local maximum function $\sup_{g \in W} \Psi(hg)  = \mathcal{M}_U^R \Psi(h)$, we obtain
\begin{eqnarray*}
 \lefteqn{\left\| \mathcal{W}_{\psi} \psi \right\|_{W^R(C^0,{\rm L}^1_{v_0})}} \\ & \preceq & 
 \int_H (1+\|h\|_\infty)^{k} \mathcal{M}_U^R (\Psi)(h) w_0(h) \frac{dh}{|{\rm det}(h)|} \\
 & \preceq &  \int_H \mathcal{M}_U^R (\Phi_{t-k})(h)  (1+\|h\|_\infty)^{2k} w_0(h) |{\rm det}(h)|^{1/2} dh \\
 & = & \| \Phi_{t-k} \|_{W^R(C^0,{\rm L}^1_{m})}~.
\end{eqnarray*}
Since $t-k \ge \ell$, the final expression is finite by assumption, and the sufficient criterion for $\psi \in \mathcal{B}_{v_0}$ is established.  

Regarding existence of compactly supported atoms, we recall from \cite[Lemma 4.1]{Fu_coorbit} the existence of a partial differential operator $D$ of order $k \le 2d$ with constant coefficients and the property that, for all 
$\rho \in {\rm L}^1(\mathbb{R}^d)$ with integrable partial derivatives of sufficiently high order, $D^t \rho$ has vanishing moments of order $t$. Thus picking $\rho \in C_c^\infty(\mathbb{R}^d)$ and letting $\psi = D^t \rho$ yields the desired compactly supported atom. 
 \hfill $\Box$

\section{Verifying strong temperate embeddedness}
\label{sect:examples}

Theorem \ref{thm:main} shows that explicit vanishing moment conditions for elements of $\mathcal{B}_{v_0}$, given a concrete dilation group $H$ and control weight $v_0$, can be obtained by the following steps: 
\begin{enumerate}
 \item Compute the dual orbit $\mathcal{O}$, and the auxiliary function $A: \mathcal{O} \to \mathbb{R}^+$.
 \item Compute an upper estimate of 
 \[ \Phi_\ell : H \ni h \mapsto \int_{\mathbb{R}^d} A(\xi) A(h^T \xi) d\xi ~,\]
 i.e., determine $\widetilde{\Phi}_\ell: H \to \mathbb{R}^+$ with $\Phi_\ell \preceq \widetilde{\Phi}_\ell$. 
 \item Determine $\Psi_\ell$ with  $\mathcal{M}_U^R \widetilde{\Phi}_\ell(\xi) \preceq \Psi_\ell(\xi)$. 
 \item Determine an explicit $\ell$ such that 
 \[
  \int_H \Psi_\ell(h) w_0(h) |{\rm det}(h)|^{-1/2} (1+\| h \|_\infty)^{2(s+d+1)} dh < \infty~.
 \]
\end{enumerate}

In this section we will see that this program can indeed be carried out for a large class of dilation groups, leading to concrete criteria. In particular, our results will cover all admissible dilation groups in dimension two. Note that there is some freedom of choice in picking the upper bounds $\widetilde{\Phi}_\ell$ and $\Psi_\ell$. 

\begin{remark} \label{rem:suboptimal}
In the following calculations, we have typically tried to cover large classes of weights with minimal computational effort, possibly at the cost of suboptimal estimates for $\ell$. Given any concrete group dilation $H$ and a control weight $v_0$ on $G$ of particular interest, better estimates for $\ell$ may be possibly achieved by other methods. 

It should be mentioned at this point that the approach via Theorem \ref{thm:main} is probably not well-suited for obtaining optimal estimates for the number of vanishing moments. Note that the theorem is based on the decay estimate in Lemma \ref{lem:decay_est_wc} which uses the global behaviour of the wavelet and its Fourier transform. It seems likely that for the estimation of local quantities, such as $\mathcal{M}_U^R(\mathcal{W}_\psi \psi)$, this is not the sharpest available method. As an interesting alternative, we mention the techniques of \cite{FeSu}, which achieve similar results with significantly less vanishing moments, but only for isotropic dilations.  
\end{remark}

\subsection{The similitude groups}

The similitude dilation group is defined by $H = \mathbb{R}^+ \cdot SO(d) \subset {\rm GL}(d,\mathbb{R})$, for $d>2$ denote the similitude group. In the interest of a unified treatment, we let $H= \mathbb{R} \setminus \{ 0 \}$ in the case $d=1$. The similitude group dilation group $H$ for $d=2$ was the first higher-dimensional dilation group used for the construction of continuous wavelet transforms \cite{AnMuVa,Mu}. The open dual orbit is easily computed as $\mathcal{O} = \mathbb{R}^d \setminus \{ 0 \}$, hence we find $\mathcal{O}^c = \{ 0 \}$. This implies 
\[
 A(\xi) = \min \left( |\xi|,\frac{1}{1+|\xi|} \right)~. 
\]
We write elements $h \in H$ as $h = r S$, with $r>0$ and $S \in SO(d)$; for $d=1$ we admit $S \in \{ \pm 1 \}$. Haar measure on $H$ is then given by $dh = \frac{dr}{|r|} dS$; here $dS$ denotes integration against Haar measure on $SO(d)$ normalized to one. The following lemma provides the central estimate of the auxiliary function $\Phi_\ell$:
\begin{lemma}
 For $h = rS \in H$ and $\ell>d$, we have
 \[
  \Phi_\ell (h) \le C \min(r^{\ell-d},r^{d-\ell})~.
 \]
\end{lemma}
\begin{prf}
 First note that by invariance of the euclidean distance under elements of $SO(d)$, we have $A(h \xi) = A(r \xi)$, and thus $\Phi_\ell(h) = \Phi_\ell  (r \cdot {\rm id}_{\mathbb{R}^d})$, and integration in polar coordinates yields
 \begin{eqnarray*}
  \Phi_{\ell}(h)  & = & \int_{\mathbb{R}^d} A(\xi)^\ell A(h^T \xi)^\ell d\xi \\
   & = & C_1 \int_0^\infty \min \left( s, \frac{1}{1+s} \right)^\ell \min \left( rs, \frac{1}{1+rs} \right)^\ell s^{d-1} ds~,
 \end{eqnarray*}
 where $C_1$ is the surface of the unit ball. 
 Let us now assume that $r \le 1$; the case $r>1$ will then be addressed using Lemma \ref{lem:Phi_ell_elementary}(b). 
 First note that 
 \[
  \min \left( s, \frac{1}{1+s} \right) = \left\{ \begin{array}{cc} s & s \le c \\ \frac{1}{1+s} & s > c \end{array} \right.  
 \]
where we used $c = \frac{\sqrt{5}-1}{2}$.  This implies 
 \begin{equation}
  s^{d-1} \min \left( s, \frac{1}{1+s} \right)^\ell \min \left( rs, \frac{1}{1+rs} \right)^\ell =  \left\{ \begin{array}{cc}
                                                         s^{2 \ell+d-1} r^\ell & 0 < s \le c \\ 
                                                         \frac{r^\ell s^{\ell+d-1}}{(1+s)^\ell} & c < s \le c/r \\
                                                         \frac{s^{d-1}}{(1+s)^\ell (1+rs)^\ell} & c/r < s 
                                                                                                 \end{array} \right.~,
 \end{equation}
and hence
\[
\Phi_\ell(r S)/C_1  = \left. \left( r^\ell \frac{s^{2 \ell+d}}{2 \ell +d} \right) \right|_{s=0}^{s=c} 
+ r^\ell \underbrace{\int_c^{c/r} \frac{s^{\ell+d-1}}{(1+s)^\ell} ds}_{=:I_1} + \underbrace{\int_{c/r}^\infty   \frac{s^{d-1}}{(1+s)^\ell (1+rs)^\ell} ds}_{=: I_2} ~.
\]
Now $\frac{s^\ell}{(1+s)^\ell} \le 1$ implies 
\[
 I_1 \le r^\ell \int_c^{c/r} s^{d-1} ds \le C_2 r^{\ell-d} 
\] with a suitable constant $C_2>0$. Furthermore, for $s>c/r$, we have
\[
 \frac{s^{d-1}}{(1+s)^\ell} \le (1+s)^{d-1-\ell} \le (1+c/r)^{d-1-\ell} \le C_3 r^{\ell+1-d}
\] which implies that
\[
 I_2 \le \int_{c/r}^\infty \frac{C_3 r^{\ell+1-d}}{(1+rs)^\ell} ds =
  C_3 r^{\ell-d} \int_{c}^\infty \frac{1}{(1+s)^\ell} ds = C_4 r^{\ell-d}~. 
\]
In summary, this yields 
\[
 \Phi_\ell(h) \le C_5 r^{\ell-d}~. 
\] If $r>1$, then 
\[
 \Phi_\ell(h) = |{\rm det}(h)|^{-1} \Phi_\ell(h^{-1}) \le r^{-d} C_5 r^{d-\ell} = C_5 r^{-\ell}~.
\]
This finishes the proof. 
\end{prf}

\begin{theorem} \label{thm:temp_embed_similitude}
 Assume that the control weight on $H$ fulfills
 \[
  w_0(h) \le (r+r^{-1})^\beta
 \] for some $\beta>0$. Then the dual orbit is $(s,w_0)$-strongly temperately embedded, with index
 $\ell = \beta + 2s + \frac{5}{2} d +3$.
\end{theorem}
\begin{prf}
 Since the mapping $\Psi_\ell: h = rS \mapsto \min(r^{d-\ell},r^{\ell-d})$ is submultiplicative, Lemma \ref{lem:max_fun_weight} implies that $c_1 \Psi_\ell \le \mathcal{M}_U^R \Psi_\ell \le c_2 \Psi_\ell$, with suitable constants $0< c_1 \le c_2$. We have $|{\rm det}(h)| = r^d$ and $\| h \|_\infty = r$, and by the previous lemma $\Phi_\ell \le C \Psi_\ell$. Hence it is sufficient to prove that the integral
 \[
  \int_{0}^\infty  (r+r^{-1})^\beta (1+r)^{2(s+d+1)} r^{-d/2} \min\left( r^{d-\ell},r^{\ell-d} \right) \frac{dr}{r}
 \] is finite. 
This is the case as soon as $\ell> \beta + 2s + \frac{5}{2}d+2$. 
\end{prf}

We will next exhibit homogeneous Besov spaces as coorbit spaces over the similitude group, a fact that has already been noted in \cite[7.2]{FeiGr0}. Since the argument for higher dimensions is only sketched in \cite{FeiGr0}, I include a short proof that combines results from \cite{Fu_coorbit} with the $\varphi$-transform characterization due to Frazier and Jawerth. 
\begin{theorem} \label{thm:Besov_as_coorbit}
For all $1 \le p,q < \infty$, we have $\dot{B}_{p,q}^{\alpha}(\mathbb{R}^d) = Co({\rm L}^{p,q}_v(G))$, with weight function
\[
 v(x,h) = v(x,rS) = r^{-\alpha-d/2+d/q}~.
\] 
\end{theorem}

\begin{prf}
 We will use the $\varphi$-transform characterization of Frazier and Jawerth, see \cite{FJW}. The $\varphi$-transform is based on the choice of two isotropic Schwartz functions $\varphi, \psi$ with Fourier transforms compactly supported away from zero, and satisfying (amongst other properties)
\[
 \forall f \in \mathcal{S}'(\mathbb{R}^d)/\mathcal{P}~:~ f = \sum_{j \in \mathbb{Z}}
 f \ast \varphi_j \ast \psi_j~,
\] with convergence in $\mathcal{S}'(\mathbb{R}^d)$ modulo polynomials. Here $\varphi_j (x) = 2^{-j} \varphi(2^{-j}x) = 2^{-j/2} (\pi(0,2^j) \varphi)(x)$, and $\psi_j$ is defined analogously.  $\mathcal{P} \subset \mathcal{S}'(\mathbb{R}^d)$ denotes the subspace of polynomials, and the series converges in $\mathcal{S}'(\mathbb{R}^d)/\mathcal{P}$.  
The associated discrete wavelet systems are then defined (in the terminology of this paper) via
\[
 \varphi_{j,k} = \pi(2^jk, 2^j) \varphi~,~\psi_{j,k} = \pi(2^jk,2^j)\psi~.
\] Here we have already somewhat adapted the notation of \cite{FJW} to the terminology of this paper, in particular the indexing conventions for small vs. large scales used here differ from \cite{FJW}. The norms of the associated discrete coefficient spaces are defined by
\[
 \left\| (c_{j,k})_{j,k} \right\|_{\dot{b}_{p,q}^\alpha} = \left(\sum_{j \in \mathbb{Z}} \left( \sum_{k \in \mathbb{Z}^d} \left( 2^{-j\alpha - jd/2+jd/p} |c_{j,k}| \right)^p \right)^{p/q} \right)^{1/q} 
\] with the usual adjustments in the cases $p=\infty$ and/or $q=\infty$. 
We now define the separated subset
\[
 Z = \{ (2^jk,2^j \cdot {\rm id}_{\mathbb{R}^d}): j \in \mathbb{Z}, k \in \mathbb{Z}^d \} \subset G
\]
which is in obvious bijective correspondence to $\mathbb{Z} \times \mathbb{Z}^d$. This bijection then induces an isometric isomorphism
 \[
 \dot{b}_{p,q}^\alpha \cong \ell^{p,q}_v (Z)~ 
\] which we will use to identify the two spaces at our convenience.

Now \cite[Theorem 6.16]{FJW} states that $f \in \dot{B}_{p,q}^\alpha(\mathbb{R}^d)$ iff $(\langle f, \varphi_{j,k} \rangle)_{j,k} \in \dot{b}_{p,q}^\alpha$, and we will use this characterization to show that $\dot{B}_{p,q}^\alpha(\mathbb{R}^d) = Co({\rm L}^{p,q}_v(G))$

To see the inclusion $\dot{B}_{p,q}^\alpha(\mathbb{R}^d) \subset Co({\rm L}^{p,q}_v(G))$, first recall by
 (\ref{eqn:cont_weight_sep}) that there exists a control weight for $L^{p,q}_v(G)$ satisfying the assumption of \ref{thm:bl_atoms}. Thus this theorem yields that both $\varphi$ and $\psi$ are in the space $\mathcal{B}_{v_0}$. 
 Moreover, the set $\{ (2^jk, 2^j) : j \in \mathbb{Z} , k \in \mathbb{Z}^d \}$ is a separated subset of $G$. Hence, if $f \in \dot{B}_{p,q}^\alpha(\mathbb{R}^d)$, the sum
\[
 \sum_{j,k} \langle f, \varphi_{j,k} \rangle \pi(2^jk,2^j) \psi 
\] converges in
$Co ({\rm L}^{p,q}_v(G))$, with respect to the norm of that space, by \cite[Theorem 6.1(ii)]{FeiGr1}. Here we used $(\langle f, \varphi_{j,k} \rangle)_{j,k} \in \dot{b}_{p,q}^\alpha$, the observation that for $Y = {\rm L}^{p,q}_V(G)$, the norm equivalence$\| (c_{j,k})_{j,k} \|_{Y_d} \asymp \| (c_{j,k})_{j,k} \|_{\dot{b}^q_{\alpha,p}}$ holds (see  Remark \ref{rem:coeff_space}), and the density of the finitely supported sequences in the coefficient space (this is why we exclude the value $\infty$ for $p$ and/or $q$). Since $Co({\rm L}^{p,q}_v(G)) \subset \mathcal{S}'(\mathbb{R}^d)/\mathcal{P}$  continuously (see the remarks following \cite[Corollary 4.5]{Fu_coorbit}), we find that the $Co({\rm L}^{p,q}_v)$-limit of the sum coincides with $f$. Hence $f \in Co({\rm L}^{p,q}_v(G))$. 

Conversely, assume that $f\in Co ({\rm L}^{p,q}_{v}(G))$, then the fact that $\varphi \in \mathcal{B}_{v_0}$ allows to invoke Theorem \ref{thm:at_dec}. Hence, for sufficiently large $m \in \mathbb{N}$, and a suitable finite set $\mathcal{S} \subset SO(d)$ of rotations (which we can assume to contain the identity matrix), the set 
\[
 Z' = \left\{ \left(2^{j/m} S \frac{k}{m}, 2^{j/m}S \right) : j, k \in \mathbb{Z}, S \in \mathcal{S} \right\} \subset G 
\] will be such that \ref{thm:at_dec}(c) applies, yielding that 
\[
 \left( \langle f, \pi(z') \varphi \rangle \right)_{z' \in Z'} \in \ell^{p,q}_v(Z')~.
\]
But since $Z \subset Z'$, this implies $(\langle f, \varphi_{j,k} \rangle)_{j,k} \in \dot{b}_{p,q}^\alpha$, and thus $f \in \dot{B}_{p,q}^\alpha(\mathbb{R}^d)$. 
\end{prf}

\begin{remark} \label{rem:proof_ex_1_2}
Let us now work out concrete vanishing moment conditions for atoms in $\dot{B}^{\alpha}_{p,q}(\mathbb{R}^d)$. 
By  (\ref{eqn:cont_weight_sep}) (with $s=0$), there exists a control weight for ${\rm L}_v^{p,q}(G)$ that is majorized by 
\begin{eqnarray*}
\lefteqn{w_0(x,rS) = }\\   & = &  \max \left(1, \Delta_G(0,rS)^{-1} \right) \left(|{\rm det}(rS)|^{1/p-1/q}+|{\rm det}(rS)|^{1/q-1/p} \right)  (r^{-\alpha-d/2+d/q}+r^{\alpha+d/2-d/q}) \\
 & \le & \left( r^{-d}+r^d \right)^2 (r^{-\alpha-d/2+d/q}+r^{\alpha+d/2-d/q}) \\
 & \preceq & (r+r^{-1})^{2 d+|\alpha-d/2+d/q|}~.
 \end{eqnarray*}
 Thus vanishing moments of order $t> |\alpha-d/2+d/q| + \frac{11}{2} d +3$ will suffice, by
 Theorems \ref{thm:temp_embed_similitude} and \ref{thm:main}. Applying this to the case $\alpha = d/2-d/q$ and $p=q$, we obtain  Theorem \ref{thm:ex_Besov} by combining Theorem \ref{thm:at_dec} with Proposition \ref{lem:nl_approx_frame}.
\end{remark}

\subsection{The diagonal groups}

The diagonal group of dimension $d$ is
\begin{equation} \label{eqn:def_diag_group}
 H = \left\{ \left( \begin{array}{cccc} a_1 &  & &  \\ & a_2 &  & \\ & & \ddots  &  \\
                     & & & a_d \end{array} \right) : \prod_{i=1}^d a_i \not= 0 \right\}~.
\end{equation}
The open dual orbit is given by $\mathcal{O} = \{ \xi = (\xi_1,\ldots,\xi_d)^T \in \mathbb{R}^d: \prod_i \xi_i \not= 0 \}$. The auxiliary function $A$ is given by 
\[
 A(\xi) = \min \left( \frac{\min_{i} |\xi_i|}{1+\sqrt{|\xi|^2- \min_{i} |\xi_i|^2}},\frac{1}{1+|\xi|} \right)~. 
\]
The control weights we are interested in are of the type 
\[  v_0(x,h) = (1+|x|+\|h\|_\infty)^s w_0(h) ~,\]
where $w_0(h)  = \prod_{i=1}^d \left(a_i+a_i^{-1}\right)^\alpha$.
The diagonal group is in fact a special case of a diagonally acting direct product group. Hence we first prove the following, somewhat more general result, which is of independent interest. 

\begin{lemma}
 Let $H_j < {\rm GL}(d_j,\mathbb{R})$ be admissible dilation groups, with $j=1,2,\ldots, k$, and $\sum_{j=1}^k d_j =d$, and let $H < {\rm GL}(d,\mathbb{R})$ be defined as
 \[
  H = \left\{ \left( \begin{array}{cccc} h_1 &  &  &  0 \\ & h_2 & & \\ & & \ddots  & \\  
                    0 & & & h_k \end{array} \right) : h_j \in H_j, j=1,\ldots,k\right\}~.
 \] Then $H$ is admissible. Let $v_0$ denote a weight on $G = \mathbb{R}^d \rtimes H$, satisfying the estimate
 \[
  v_0(x,h) = (1+|x|+\|h\|_\infty)^s \prod_{j=1}^k w_j(h_j)  ~, h =\left( \begin{array}{cccc} h_1 &  & & 0 \\ & h_2 & & \\ & & \ddots  & \\  
                     0 & & & h_k \end{array} \right)~,
 \] with weights $w_j$ on $H_j$. Let $\mathcal{O}_j$ denote the open dual orbit of $H_j$, and $\mathcal{O}$ the open dual orbit of $H$. If $\mathcal{O}_j$ is $(s,w_j)$-temperately embedded with index $\ell_j$, for $j=1,2,\ldots,k$, then $\mathcal{O}$ is $(s,w)$-temperately embedded with index $\ell = k \max(\ell_1,\ell_2,\ldots,\ell_k)$. 
\end{lemma}
\begin{prf}
 Obviously, $H$ has a unique open orbit given by $\mathcal{O} = \prod_{j=1}^k \mathcal{O}_j $, and the associated fixed groups are the direct products of the fixed groups in $H_j$, respectively. These are compact by assumption, hence $H$ is admissible as well.  
 
 Let $A: \mathcal{O} \to \mathbb{R}^+$ and $A_j: \mathcal{O}_j \to \mathbb{R}^+$ ($j=1,\ldots,k$) denote the associated envelope functions. Then we have, for $\xi = (\xi_1,\xi_2,\ldots,\xi_k)$ with $\xi_j \in \mathcal{O}_j$: 
 \begin{equation} \label{eqn:est_A_prod}
  A(\xi)^k \le \prod_{j=1}^k A_j(\xi_j)~.
 \end{equation} To see this, let  
 $\eta_j \in \mathcal{O}_j^c \subset\mathbb{R}^{d_j}$ denote elements of minimal distance to $\xi_j$, i.e. 
 $|\xi_j-\eta_j| = {\rm dist}(\xi_j, \mathcal{O}_j^c)$. 
 Let $j_0 \in \{ 1,\ldots, k \}$ be arbitrary, and let $\eta = (\xi_1,\ldots,\xi_{j_0-1},\eta_{j_0},\xi_{j_0+1},\ldots,\xi_k)^T \in \mathcal{O}^c$. Then we have   
 \[ {\rm dist}(\xi,\mathcal{O}^c) \le |\xi-\eta| = |\xi_{j_0} - \eta_{j_0}| ~,
 \]
 as well as 
 \[
  \sqrt{|\xi|^2-{\rm dist}(\xi,\mathcal{O}^c)^2} \ge \sqrt{|\xi|^2-|\xi-\eta|^2} = \sqrt{|\xi_{j_0}|^2-|\eta_{j_0}|^2}~.
 \]
Combining the two estimates, we get  
 \[
  \frac{{\rm dist}(\xi,\mathcal{O}^c)}{1+\sqrt{|\xi|^2-{\rm dist}(\xi,\mathcal{O}^c)^2}} 
  \le \frac{|\xi_{j_0}-\eta_{j_0}|}{1-\sqrt{|\xi_{j_0}|^2-|\eta_{j_0}|^2}} \le A_{j_0}(x_{j_0})~.
 \]
Together with the obvious estimate $\frac{1}{1+|\xi|} \le \frac{1}{1+|\xi_{j_0}|}$, this yields
$A(\xi) \le A_{j_0}(\xi_{j_0})$, and taking the product over $j_0=1,\ldots,k$ yields (\ref{eqn:est_A_prod}). 

Now Fubini's theorem immediately implies that  $\Phi_{\ell k}(h_1,\ldots,h_k) \le \prod_{j=1}^k \Phi_{\ell,j}(h_j)$, where $\Phi_{m} : H \to \mathbb{R}^+ \cup \infty$ and $\Phi_{\ell,j}: H_j \to \mathbb{R}^+ \cup \{ \infty \}$ are the functions defined according to (\ref{eqn:def_Phi_ell}). Note that here we used that left Haar measure on $H$ is the product of left Haar measures on the $H_j$. Now the assumptions on the $H_j$, together with 
$(1+\|h \|_\infty)^s  \le \prod_{j=1}^k (1+\| h_j \|_\infty)^s$, yield the desired result.  
 \end{prf}

 We can now combine the lemma with Theorem \ref{thm:temp_embed_similitude}, to obtain the following result.
\begin{corollary} \label{cor:diag_group}
Let $H$ be defined by (\ref{eqn:def_diag_group}), and assume that the control weight on $H$ fulfills
 \[
  w_0(h) \le \prod_{a=1}^d (|a_j|+|a_j|^{-1})^\alpha
 \] for some $\alpha>0$. Then the dual orbit is $(s,w_0)$-strongly temperately embedded, with index
 $\ell = d (\alpha + 2s + 11/2)$.
\end{corollary}

\subsection{Shearlet groups in dimension two}

Fix a real parameter $c$, and let
\[
H =  H_c = \left\{ \left( \begin{array}{cc} a & b \\ 0 & a^c \end{array}
 \right)~:~a, b \in \mathbb{R}, a\not= 0 \right\}~.
\] Here we use the convention $a^c = {\rm sign}(a) |a|^c$ for $a < 0$. 
For $c=1/2$, $H_c$ is the shearlet group introduced in
\cite{KuLa}, and further studied (e.g.) in \cite{DaKuStTe,DaStTe10}, see also \cite{shearlet_book} for a comprehensive overview. The other groups are obviously closely related; the parameter $c$ can be understood as controlling the anisotropy used in the scaling.
 Haar measure on $H$  is given by $db \frac{da}{|a|^{2}}$, the modular function is $\Delta_H(h) = |a|^{c-1}$. The dual orbit is
computed as
\[ \mathcal{O} =
 \mathbb{R}^2 \setminus (\{ 0 \} \times \mathbb{R}))~.\]
For $h = \left( \begin{array}{cc} a & b \\ 0 & a^c
\end{array}
 \right) \in H$ and $\xi_0 = (1,0)^T \in \mathcal{O}$, we obtain $h^T \xi_0 = (a,b)^T$.

 One computes 
 \[
 A(\xi) = \min \left( \frac{|\xi_1|}{1+|\xi_2|},\frac{1}{1+|\xi|}\right)~. 
 \]
  
 We now come to the critical Step 2 in the program outlined at the beginning of this section. The following estimate is central to this subsection: 
 \begin{lemma} \label{lem:est_shearlet_coeff}
  Let $\Phi_\ell: H \to \mathbb{R}^+$ be defined according to  (\ref{eqn:def_Phi_ell}).
 Assume that  $t, r_1,r_2 \ge 1$ are positive integers satisfying $r_2>1$ as well as 
\begin{equation} \label{eqn:Phi_ell_est_cond}
 t \ge 3 r_1+ (3+6|c|) r_2 + 6|c|+2~. 
\end{equation}
 Then, for all $h= \left( \begin{array}{cc} a & b \\ 0 & a^c \end{array} \right)  \in H$, we have the estimate
\begin{equation}  \label{eqn:Phi_ell_est}
 \Phi_t(h) \le C (|a|+|a|^{-1})^{-r_1} (1+|b|)^{-r_2}~,
\end{equation}
with a suitable constant $C>0$.
 \end{lemma}
\begin{prf}
We first prove a somewhat stronger estimate for $|a|\le 1$, and then use Lemma \ref{lem:Phi_ell_elementary}(b) for the other case. 
Let 
\begin{equation}
 u_1 = r_1 + (1+2|c|)r_2+2|c|~,~ u_2 = 2u_1+2 =  2 r_1+(2+4|c|) r_2 + 4|c| +2~.
\end{equation}
Then we have $u_1 + u_2 \le t$, and thus we obtain, for all $1< u_3 \le u_1$
\begin{equation} \label{eqn:est_A_ell}
 \forall \xi = (\xi_1,\xi_2)^T \in \mathcal{O}~:~A(\xi)^t \le \frac{|\xi_1|^{u_1}}{(1+|\xi_1|)^{u_2}}
 \frac{1}{(1+|\xi_2|)^{u_3}}~. 
\end{equation}
For $\xi \in \mathcal{O}$, we have $h^T (\xi) = (a\xi_1, b \xi_1+ a^c \xi_2)^T$, and hence plugging (\ref{eqn:est_A_ell}) into the definition of $\Phi_t$ yields
\begin{eqnarray} \nonumber 
 \Phi_t(h) & = & \int_{\mathbb{R}^2} A(\xi)^t A(h^T \xi)^t d\xi \\  \nonumber 
 & \le & \int_{\mathbb{R}}  \frac{|\xi_1|^{u_1}}{(1+|\xi_1|)^{u_2}}  \frac{|a \xi_1|^{u_1}}{(1+|a \xi_1|)^{u_2}} \\ \label{eqn:zwisch1}
 & & ~~\int_{\mathbb{R}} (1+|\xi_2|)^{-u_3} (1+|b \xi_1 +a^c \xi_2|)^{-u_3} d\xi_2 d \xi_1 ~.
\end{eqnarray}
We now employ the estimate
\begin{equation}
 \label{eqn:DaStTe} \int_{\mathbb{R}} (1+|y|)^{-r} (1+\alpha |x-y|)^{-r} dy \le
 C\left(\alpha^{-1} (1+|x|)^{-r} + (1+\alpha|x|)^{-r} \right)~,
\end{equation}
(see \cite[Lemma 3.1]{DaStTe11}) with $r=u_3$, $\alpha = |a|^c$ and $x = a^{-c} b \xi_1$, to continue from 
(\ref{eqn:zwisch1}) to obtain
\[
 \Phi_t(h) \le C(I_1+I_2)
 \]
 with
 \[
  I_ 1 =  \int_{\mathbb{R}}  \frac{|\xi_1|^{u_1}}{(1+|\xi_1|)^{u_2}}  \frac{|a \xi_1|^{u_1}}{(1+|a \xi_1|)^{u_2}} \frac{|a|^{-c}}{(1+|a^{-c} b \xi_1|)^{u_3}} d\xi_1 
 \] and
 \[
  I_2 = \int_{\mathbb{R}}  \frac{|\xi_1|^{u_1}}{(1+|\xi_1|)^{u_2}}  \frac{|a \xi_1|^{u_1}}{(1+|a \xi_1|)^{u_2}}\frac{1}{(1+| b \xi_1|)^{u_3}} d\xi_1 ~.
 \]
Now
\begin{eqnarray*}
 I_1 & = & |a|^{u_1-c} \int_{\mathbb{R}} \frac{|\xi_1|^{2u_1}}{(1+|\xi_1|)^{u_2}(1+|a\xi_1|)^{u_2}
 (1+|a^{-c}b\xi_1|)^{u_3}} d\xi_1 \\
 & = & |a|^{u_1-c} |a^{-c} b|^{-u_3} \int_{\mathbb{R}}
  \frac{|\xi_1|^{2u_1-u_3}}{(1+|\xi_1|)^{u_2}(1+|a\xi_1|)^{u_2}} \underbrace{\frac{|a^{-c} b \xi_1|^{u_3}}{(1+|a^{-c} b \xi_1|)^{u_3}}}_{\le 1} d\xi_1 \\
  & \le & |a|^{u_1-c+c u_3} |b|^{-u_3} C \int_{\mathbb{R}} \frac{\max\{ 1,|\xi_1|^{2u_1}}{(1+|\xi_1|)^{u_2}} d\xi_1~,
\end{eqnarray*}
where we used that $2u_1-u_3 \ge 2$. 
Our choice of constants implies $u_2 - 2u_1 = 2$, hence the integral converges.  Since this holds for all $1<u_3 \le r_2$, we thus obtain, using $|a| \le 1$,
\begin{equation} \label{eqn:est_I_1}~
 I_1 \preceq |a|^{u_1-|c|(1+r_2)} (1+|b|)^{-r_2}~. 
\end{equation}
For $I_2$, we obtain with a similar calculation
\begin{eqnarray*}
 I_2 & = & |a|^{u_1} |b|^{-u_3} \int_{\mathbb{R}} \frac{|\xi_1|^{2u_1-u_3}}{{(1+|\xi_1|)^{u_2}(1+|a\xi_1|)^{u_2}}} \underbrace{\frac{|b \xi_1|^{u_3}}{(1+|b \xi_1|)^{u_3}}}_{\le 1} d\xi_1 \\
 & \preceq & |a|^{u_1} |b|^{-u_3}~,
\end{eqnarray*}
and since $1< u_3 \le r_2$ was arbitrary, we find for $|b| \ge 1$: 
\[
 I_2 \preceq |a|^{u_1} (1+|b|)^{-r_2}~.
\]
But this means that for the case $|a|\le 1$, $|b| \ge 1$ we have in fact established
\begin{equation} \label{eqn:est_Phi_ell_HM}
 \Phi_t(h) \preceq |a|^{u_1-|c|(1+r_2)} (1+|b|)^{-r_2}~,
\end{equation} which is stronger than (\ref{eqn:Phi_ell_est}), since $u_1 \ge r_1+|c|(1+r_2) $. 

To establish (\ref{eqn:est_Phi_ell_HM}) for $|a|\le 1$ and $|b| \le 1$, we employ similar (but easier) estimates to derive for $j=1,2$
\[
 I_j \preceq |a|^{u_1-c} \preceq |a|^{u_1-|c|(1+r_2)} (1+|b|)^{-r_2}~.
\]

In order to apply Lemma \ref{lem:Phi_ell_elementary} to the case $|a|>1$, we first compute $h^{-1} = \left( \begin{array}{cc} a^{-1} & -a^{-c-1}b \\ 0 & a^{-c} \end{array} \right)$.  
Hence Lemma \ref{lem:Phi_ell_elementary} and (\ref{eqn:est_Phi_ell_HM}) yield
\begin{eqnarray*}
 \Phi_t(h) & = & \underbrace{|{\rm det}(h)|^{-1}}_{=|a|^{-1-c}} \Phi_t(h^{-1}) \\
 & \preceq &  |a|^{-1-c}|a|^{-u_1+|c|(1+r_2)} (1+|a^{-c-1}b|)^{-r_2}  \\
 & \le & |a|^{-u_1+|c|(2+r_2)} |a|^{(1+|c|)r_2} (1+|b|)^{-r_2} \\
 & = & |a|^{-u_1+|c|(2+r_2)+(1+|c|)r_2} (1+|b|)^{-r_2} \\
 & = & |a|^{-r_1} (1+|b|)^{-r_2}~,
\end{eqnarray*}
by choice of $u_1$. This proves the Lemma. 
\end{prf}

\begin{remark}
 Combining Lemma \ref{lem:decay_est_wc} with (\ref{eqn:Phi_ell_est}) yields a decay estimate for shearlet coefficients, that is possibly of independent interest:
 \[
  |\mathcal{W}_{\psi} f (x,h)| \le C (1+|x|)^{-m} (|a|+|a|^{-1})^{-r_1} (1+|b|)^{-r_2}~,
 \]
which holds for all shearlets $\psi$ and signals $f$ fulfilling sufficient vanishing moment, smoothness and decay conditions. Note that the additional factors $(1+\|h \|_{\infty})^m$ and $|{\rm det}(h)|^{1/2}$ occurring in 
Lemma \ref{lem:decay_est_wc} are absorbed by suitably high powers of $(|a|+|a|^{-1})^{-1}(1+|b|)^{-1}$. 
\end{remark}

The following theorem establishes strong temperate embeddedness. We use the same $w_0$  as in \cite{DaKuStTe}; note that the parametrization of the dilation group used in that paper differs from the one employed here.  
\begin{theorem}
 \label{thm:shearlet_str_embed}
 Let $w_0(h) = (|a|+|a|^{-1})^{u_1} (|a|+|a|^{-1}+|a^c b|)^{u_2}$, for some $u_1,u_2>0$. Then $\mathcal{O}$ is strongly $(s,w_0)$-temperately embedded with index 
 \begin{equation} \label{eqn:ind_te_shearlet}
  \ell = \left\lceil 3 u_1 + (9+9|c|)u_2 + 18(1+|c|) s + \frac{147}{2} |c| + \frac{163}{2} \right\rceil ~. 
 \end{equation}
\end{theorem}
\begin{prf}
Introduce 
\[
 r_2 = u_2+2s+8
\]
and 
\begin{eqnarray*}
 r_1 & =&  2 + r_2 + u_1 + u_2(1+|c|) + 2(1+|c|) s + \frac{15}{2} |c| + \frac{17}{2} \\
  & = & u_1 + u_2 (2+|c|) + 2(2+|c|) s + \frac{13}{2}|c| + \frac{37}{2}~.
\end{eqnarray*}
Then one readily verifies
\[
 \ell \ge 3 r_1+ (3+6|c|) r_2 + 6|c|+2~,
\] hence Lemma \ref{lem:est_shearlet_coeff} is applicable. 

 We fix the neighborhood $U$ of the identity element in $H$ as 
 \[
  U = \left\{ \left( \begin{array}{cc} a_1 & b_1 \\ 0 & a_1^c \end{array} \right) : 1/2 < a_1 < 2~, ~|b_1|< 1 \right\}~.
 \]
Given $h = \left( \begin{array}{cc} a & b \\ 0 & a^c \end{array} \right) \in H$, we employ (\ref{eqn:Phi_ell_est}) to estimate
\begin{eqnarray*} 
 \mathcal{M}_U^R(\Phi_t)(h) & \preceq & \sup \left\{ \left(|aa_1|+|aa_1|^{-1} \right)^{-r_1} (1+|ab_1
 +a_1^c b|)^{-r_2} ~:~1/2 < a_1 < 2, |b_1|< 1 \right\} \\
 & \le & 2^{r_1} \left(|a|+|a|^{-1} \right)^{-r_1} \sup_{1/2 < a_1 < 2, |b_1|\le 1} (1+|ab_1
 +a_1^c b|)^{-r_2} \\
 & \preceq & \left(|a|+|a|^{-1}\right)^{-r_1} (1+|a|)^{r_2} (1+|b|)^{-r_2} \\
 & \preceq & \left(|a|+|a|^{-1}\right)^{r_2-r_1} (1+|b|)^{-r_2}~.
\end{eqnarray*}
A left Haar measure on $H$ is given by $db \frac{da}{|a|^{2}}$, and instead of the operator norm, we take the equivalent norm $\| h \| = |a|+|a^c|+|b|$.  Hence we need to estimate
\begin{eqnarray*}
 I & = & \int_{\mathbb{R}'} \int_{\mathbb{R}}  (|a|+|a|^{-1})^{u_1} (|a|+|a|^{-1}+|a^c b|)^{u_2} |a|^{-(1+c)/2}    \\  & &  (1+|a|+|a|^c+|b|)^{2(s+3)}  \left(|a|+|a|^{-1}\right)^{r_2-r_1} (1+|b|)^{-r_2} db \frac{da}{|a|^{2}}  ~.
\end{eqnarray*}
Using the estimates
\[
 (|a|+|a|^{-1}+|a^c b|) \le (|a|+|a|^{-1})^{1+|c|} (1+|b|)
\]
and
\[
 (1+|a|+|a|^c +|b|) \preceq (|a|+|a|^{-1})^{1+|c|}(1+|b|)~, 
\] we find that 
\[
 I \preceq \int_{\mathbb{R}} (|a|+|a|^{-1})^{e_1} da \int_{\mathbb{R}} (1+|b|)^{e_2} db~,
\]
with exponents
\begin{eqnarray*}
 e_1 & = &  r_2 - r_1 + u_1 + u_2(1+|c|) + 2(s+3)(1+|c|) + \frac{1}{2} |c| + \frac{5}{2} \\
  & = & r_2 - r_1 + u_1 + u_2 (1+|c|) + 2(1+|c|) s + \frac{13}{2} |c| + \frac{17}{2}
\end{eqnarray*}
and 
\[
 e_2 = -r_2 + u_2 + 2s + 6 ~. 
\]
Our choice of $r_1,r_2$ implies that $e_1=e_2 = -2$. Hence $I < \infty$, and we are done.
\end{prf}

 \begin{remark} \label{rem:van_mom_shearlet}
  In order to formulate an analog of Theorem \ref{thm:ex_Besov}, we want to explicitly determine a sufficient number of vanishing moments for atoms in $Co({\rm L}^p(G))$. For this purpose, first observe that for $h = \left( \begin{array}{cc} a & b \\ 0 & a^c \end{array} \right)$, the modular function of $G$ given by $\Delta_G(x,h) = \frac{\Delta_H(h)}{|{\rm det}(h)|} = |a|^{-2}$. Hence a control weight for $Co({\rm L}^p(G))$ is majorized by 
  \[
   w_0(x,h)  =  \max \left(\Delta_G(0,h)^{-1}, \Delta_G(0,h) \right) \preceq (|a|+|a|^{-1})^{2}~.   
  \] Hence Theorem \ref{thm:main}, together with the formula from \ref{thm:shearlet_str_embed} (with $d=2, c=1/2,u_1 = 2,u_2 = s= 0$) yields that for $k\ge 127$, any function $\psi \in {\rm L}^2(\mathbb{R}^2)$ with integrable derivatives of order up to $2k$ and vanishing moments of order $k$ on the coordinate axis $\{ 0 \} \times \mathbb{R}$ will be an atom for all coorbit spaces $Co({\rm L}^p(G))$, $1 \le p \le 2$. Thus a shearlet analog of \ref{thm:ex_Besov} follows from Theorem \ref{thm:at_dec}. 
 \end{remark}
 
 \begin{remark}
  For dimension $d=2$, the examples covered in this section contain a system of representatives of admissible dilation groups up to conjugacy and finite index, see \cite{FuDiss,FuCuba}. It is easy to see that passing from an admissible dilation group $H_1$ to a conjugate group $H_2 = gH_1 g^{-1}$, the open dual orbit of $H_1$ is strongly temperately embedded with respect to a weight $w_1$ on $H_1$ if and only if the open dual orbit $H_2$ is strongly temperately embedded with respect to the weight $w_2: H_2 \ni h_2 \mapsto w_1(g^{-1} h_2 g)$. Also, it is clear that the finite index condition has no effect on this property. Thus we have indeed verified that for all reasonable choices of weights on an admissible dilation group in dimension two, the condition of Theorem \ref{thm:main} is fulfilled.   
 \end{remark}

\section{Acknowledgement}
I thank Hans Feichtinger for useful comments and references, and Felix Voigtlaender for a thorough reading of an earlier draft, and many useful comments and suggestions for improvement. 
 

\bibliography{atom_van_moments.bib}
\bibliographystyle{plain}
\end{document}